\documentclass[12pt]{amsart}
\usepackage[utf8]{inputenc}
\usepackage[english]{babel}
\usepackage{amsmath, amssymb, amsthm, amscd,color,comment}
\usepackage{cancel}
\usepackage{cite}
\usepackage{alltt}
\usepackage[dvipsnames]{xcolor}
\usepackage{array}
\usepackage[small,bf,labelsep=period]{caption}
\usepackage{mathtools}
\usepackage{hyperref}

\makeatletter
\def\moverlay{\mathpalette\mov@rlay}
\def\mov@rlay#1#2{\leavevmode\vtop{%
   \baselineskip\z@skip \lineskiplimit-\maxdimen
   \ialign{\hfil$\m@th#1##$\hfil\cr#2\crcr}}}
\newcommand{\charfusion}[3][\mathord]{
    #1{\ifx#1\mathop\vphantom{#2}\fi
        \mathpalette\mov@rlay{#2\cr#3}
      }
    \ifx#1\mathop\expandafter\displaylimits\fi}
\makeatother

\newcommand{\cupdot}{\charfusion[\mathbin]{\cup}{\cdot}}
\newcommand{\bigcupdot}{\charfusion[\mathop]{\bigcup}{\cdot}}

\usepackage{pgfplots}
\pgfplotsset{compat=1.15}
\usepackage{mathrsfs}
\usetikzlibrary{arrows}

\usepackage{enumerate}
\usepackage{cancel}
\newtheorem{theorem}{Theorem}[section]

\newtheorem{thmx}{Theorem}

\newtheorem{example}[theorem]{Example}

\newtheorem{remark}[theorem]{Remark}
\newtheorem{lemma}[theorem]{Lemma}

\newtheorem{proposition}[theorem]{Proposition}

\DeclarePairedDelimiter\ceil{\lceil}{\rceil}
\DeclarePairedDelimiter\floor{\lfloor}{\rfloor}

\def\la{\lambda}

\def\N{\mathbb N}
\def\R{\mathbb R}
\def\Z{\mathbb Z}

\def\cP{\mathcal P}

\def\cX{\mathcal X}
\def\cY{\mathcal Y}

\def\a{\alpha}

\def\fqs{\mathbb F_{q^2}}

\def\fq{\mathbb F_q}

\def\deg{{\rm deg}}

\def\Ap{{\rm Ap}}

\def\a{\alpha}
\def\char{\mbox{\rm Char}}

\definecolor{egreen}{RGB}{19, 126, 8}

\newcommand{\al}{\alpha}
\newcommand{\be}{\beta}

\begin{document}

\title{On Kummer extensions with one place at infinity}



\author{Erik A. R. Mendoza}

\address{Instituto de Matemática, Universidade Federal do Rio de Janeiro, Cidade Universitária, CEP 21941-909, Rio de Janeiro, Brazil}
\email{erik@im.ufrj.br}

\thanks{The research of Erik A. R. Mendoza was partially supported by FAPERJ/RJ-Brazil (Grant 201.650/2021).}

\begin{abstract} 
Let $K$ be the algebraic closure of $\fq$. We provide an explicit description of the Weierstrass semigroup $H(Q_\infty)$ at the only place at infinity $Q_{\infty}$ of the curve $\cX$ defined by the Kummer extension with equation $y^m=f(x)$, where $f(x)\in K[x]$ is a polynomial satisfying $\gcd (m, \deg f)=1$. As a consequence, we determine the Frobenius number and the multiplicity of $H(Q_{\infty})$ in some cases, and we discuss sufficient conditions for the Weierstrass semigroup $H(Q_{\infty})$ to be symmetric. Finally, we characterize certain maximal Castle curves of type $(\cX, Q_{\infty})$.
\end{abstract}

\maketitle

\section{Introduction}
Let $K$ be the algebraic closure of the finite field $\fq$ with $q$ elements. Consider $\cX$ a nonsingular, projective, absolutely irreducible algebraic curve over $K$ with genus $g(\cX)$ and denote by  $K(\cX)$ its function field.  For a function $z \in K(\cX)$, we let $(z), (z)_\infty$ and $(z)_0$ stand for the principal, pole and zero divisor of the function $z$ in $K(\cX)$ respectively. 

Given a place $Q$ in the set of places $\cP_{K(\cX)}$ of the function field $K(\cX)$, the {\it Weierstrass semigroup} associated to the place $Q$ is given by
$$
H(Q):=\{s\in \N_0: (z)_{\infty}=sQ\text{ for some }z\in K(\cX)\},
$$
the complementary set $G(Q):=\N\setminus H(Q)$ is called the {\it gap set} at $Q$, and the Weierstrass Gap Theorem \cite[Theorem 1.6.8]{S2009} states that if $g(\cX) >0$, then there exist exactly $g(\cX)$ gaps at $Q$
$$
G(Q)=\{1=i_1<i_2<\dots<i_{g(\cX)}\leq 2g(\cX)-1\}.
$$
The smallest nonzero element of $H(Q)$ is called the multiplicity of $H(Q)$ and is denoted by $m_{H(Q)}$, the largest element of $G(Q)$ is called the Frobenius number and is denoted by $F_{H(Q)}$, and we say that the Weierstrass semigroup $H(Q)$ is symmetric if $F_{H(Q)}=2g(\cX)-1$.

The knowledge of the inner structure of the Weierstrass semigroup $H(Q)$ at one place in the function field $K(\cX)$ has various applications in the area of algebraic curves over finite fields. Among the most interesting ones we have the construction of algebraic geometry codes with good parameters, see \cite{MP2020}; the determination of the automorphism group of an algebraic curve, see \cite{MXY2016}; to decide if a place is Weierstrass, see \cite{ABQ2019}, and obtain upper bounds for the number of rational places (places of degree one) of a curve, such as the Lewittes bound \cite{L1990} which establishes that the number $\#\cX(\fq)$ of $\fq$-rational places of a curve $\cX$ defined over $\fq$ is upper bounded by 
\begin{equation}\label{LB}
\# \cX (\fq)\leq qm_{H(Q)}+1,
\end{equation}
where $Q$ is an $\fq$-rational place of $\cX$. The best-known upper bound for the number of $\fq$-rational places is the Hasse-Weil bound
$$
\# \cX(\fq)\leq q+1+2g(\cX)\sqrt{q},
$$
and a curve is called $\fq$-maximal if equality holds in the Hasse-Weil bound. 

A pointed algebraic curve $(\cX, Q)$ over $\fq$, where $Q$ is an $\fq$-rational place of $\cX$, is called a {\it Castle curve} if the semigroup $H(Q)$ is symmetric and equality holds in (\ref{LB}). Castle curves were introduced in \cite{MST2009} and have been studied due to their interesting properties related to the construction of algebraic geometry codes with good parameters and its duals, see \cite{MST2008, MST2009}.

Abdón, Borges, and Quoos \cite{ABQ2019} provided an arithmetical criterion to determine if a positive integer is an element of the gap set of $H(Q)$, where $Q$ is a totally ramified place in a Kummer extension defined by the equation $y^m=f(x)$, $f(x)\in K[x]$. As a consequence, they explicitly described the semigroup $H(Q)$ when $f(x)$ is a separable polynomial. This description was generalized by Castellanos, Masuda, and Quoos \cite{CMQ2016}, where they study the Kummer extension defined by $y^m=f(x)^{\lambda}$, where $\la\in \N$ and $f(x)\in K[x]$ is a separable polynomial satisfying $\gcd(m, \lambda \deg f)=1$. 

For a general Kummer extension with one place at infinity
\begin{equation}\label{curveXX}
\cX: \quad y^{m}=\prod_{i=1}^{r} (x-\a_i)^{\lambda_i}, \quad \la_i\in \N, \quad \text{and} \quad 1\leq \la_i < m,
\end{equation}
where $m\geq 2$ and $r\geq 2$ are integers such that $\gcd (m, q)=1$, $\al_1, \dots, \al_r\in K$ are pairwise distinct elements, $\la_0:=\sum_{i=1}^{r}\la_i$, and $\gcd (m, \la_0)=1$, the Weierstrass semigroup $H(Q_{\infty})$ at the only place at infinity $Q_{\infty}$ of $\cX$ was explicitly described in the following particular cases:
\begin{enumerate}[i)]
\item For $\la_1=\la_2=\dots=\la_r$, see \cite[Theorem 3.2]{CMQ2016}.
\item For any $\la_1$ and $\la_2=\la_3=\dots=\la_r=1$, see \cite[Remark 2.8]{TT2019}.
\end{enumerate}
This article aims to explicitly describe the Weierstrass semigroup $H(Q_{\infty})$ in the general case, that is, we determine the Weierstrass semigroup at the only place at infinity of the curve $\cX$ given in (\ref{curveXX}). Moreover, we provide a system of generators for the semigroup $H(Q_{\infty})$ and, as a consequence, we obtain interesting results including the following theorems:
\begin{thmx}[see Theorem \ref{teo_2}] Let $F_{H(Q_{\infty})}$ be the Frobenius number of the semigroup $H(Q_{\infty})$. Then
$$F_{H(Q_{\infty})}=m(r-1)-\la_0\text{ and }H(Q_{\infty})\text{ is symmetric}\quad  \Leftrightarrow \quad\la_j\mid m \text{ for each }j=1, \dots, r.
$$
\end{thmx}
\begin{thmx}[see Theorem \ref{teo_sym}] 
Suppose that $\gcd (m, \la_j)=1$ for each $j=1, \dots, r$. Then the following statements are equivalent:
\begin{enumerate}[i)]
\item $H(Q_{\infty})=\langle m, r \rangle$.
\item $\la_1=\la_2=\dots =\la_r$.
\end{enumerate}
If in addition $r<m$ then all these statements are equivalent to the following one:
\begin{enumerate}[i)]
\item[iii)] $H(Q_{\infty})$ is symmetric. 
\end{enumerate}
\end{thmx}
\begin{thmx}[see Theorem \ref{theorema3}]
Suppose that $\cX$ is defined over $\fqs$, $\gcd(m, \la_j)=1$ for $j=1, \dots, r$ and $r<m$. Then
$$
(\cX, Q_{\infty}) \text{ is } \fqs\text{-maximal Castle curve}\Leftrightarrow \cX \text{ is } \fqs\text{-maximal, }\la_1=\dots=\la_r,\text{ and }m=q+1.
$$
\end{thmx}
This paper is organized as follows. In Section $2$ we introduce the preliminaries and notation that will be used throughout this paper. In Section $3$ we present the main result of this paper which gives the explicit description of the semigroup $H(Q_{\infty})$ (see Theorem \ref{theorem1}). In Section $4$ we provide an explicit description of the gap set $G(Q_{\infty})$ (see Proposition \ref{prop_1}), we study the Frobenius number and the multiplicity of the semigroup $H(Q_{\infty})$ establishing a relationship between them (see Proposition \ref{prop_3}), and provide sufficient conditions for the semigroup $H(Q_{\infty})$ to be symmetric (see Theorems \ref{teo_2} and \ref{teo_sym}). In Section $5$, we characterize certain $\fqs$-maximal Castle curves of type $(\cX, Q_{\infty})$ (see Theorem \ref{theorema3}).

\section{Preliminaries and notation}
Throughout this article, we let $q$ be the power of a prime $p$, $\fq$ the finite field with $q$ elements, and $K$ the algebraic closure of $\fq$. For $a$ and $b$ integers, we denote by $(a, b)$ the greatest common divisor of $a$ and $b$, and by $b \mod{a} $ the smallest non-negative integer congruent with $b$ modulo $a$. For $c\in \R$, we denote by $\floor*{c}$, $\ceil*{c}$ and $\{c\}$ the floor, ceiling and fractional part functions of $c$ respectively. Moreover, to differentiate standard sets from multisets (that is, sets that can contain repeated occurrences of elements), we use the usual symbol $`\{\}$' for standard sets and the symbol $`\{\!\!\{\}\!\!\}$' for multisets. For a multiset $M$, the set of distinct elements of $M$ is called the support of $M$ and is denoted by $M^*$, the number of occurrences of an element $x\in M^*$ in the multiset $M$ is called the multiplicity of $x$ and is denoted by $m_M(x)$, and the cardinality of the multiset $M$ is defined as the sum of the multiplicities of all elements of $M^*$. We say that two multisets $M_1$ and $M_2$ are equal if $M_1^*=M_2^*$ and $m_{M_1}(x)=m_{M_2}(x)$ for each $x$ in the support.

\subsection{Numerical semigroups} A numerical semigroup is a subset $H$ of $\N_0$ such that $H$ is closed under addition, $H$ contains the zero, and the complement $\N_0\setminus H$ is finite. The elements of $G:=\N_0\setminus H$ are called the gaps of the numerical
semigroup $H$ and $g_H:=\#G$ is its genus. The largest gap is called the Frobenius number of $H$ and is denoted by $F_{H}$. The smallest nonzero element of $H$ is called the multiplicity of the semigroup and is denoted by $m_H$. The numerical semigroup $H$ is called symmetric if $F_H=2g_H-1$. Moreover, we say that the set $\{a_1, \dots, a_d\}\subset H$ is a system of generators of the numerical semigroup $H$ if
$$
H=\langle a_1, \dots, a_d\rangle:=\{t_1a_1+\cdots + t_da_d: t_1, \dots, t_d\in \N_0\}.
$$ 
We say that a system of generators of $H$ is a minimal system of generators if none of its proper subsets generates the numerical semigroup $H$. The cardinality of a minimal system of generators is called the embedding dimension of $H$ and will be denoted by $e_{H}$.

Let $n$ be a nonzero element of the numerical semigroup $H$. The
Apéry set of $n$ in $H$ is defined by
$$
\Ap (H,n) := \{s\in  H : s-n \notin H\}.
$$
It is known that the cardinality of $\Ap (H,n)$ is $n$. Moreover, several important results are associated with the Apéry set.
\begin{proposition}\cite[Proposition 2.12]{RG2009}\label{prop2_pre}
Let $H$ be a numerical semigroup and $S\subseteq H$ be a subset that consists of $n$ elements that form a complete set of representatives for the congruence classes of $\Z$ modulo $n\in H$. Then
$$
S=\Ap (H, n)\quad \text{if and only if}\quad g_H=\sum_{a\in S}\floor*{\frac{a}{n}}.
$$
\end{proposition}
\begin{proposition}\cite[Proposition 4.10]{RG2009}\label{prop3_pre}
Let $H$ be a numerical semigroup and $n$ be a nonzero element of $H$. Let $\Ap (H,n) = \{a_0 < a_1 <\dots< a_{n-1}\}$ be the Apéry set of $n$ in $H$. Then $H$ is symmetric if and only if 
$$
a_i+a_{n-1-i}= a_{n-1} \text{ for each }i=0, \dots, n-1.
$$
\end{proposition}
On the other hand, the following result characterizes the elements of a numerical semigroup generated by two elements and will be useful in this paper.
\begin{proposition}\cite[Lemma 1]{R2005}\label{prop_pre}
Let $x\in \Z$ and let $n_1, n_2\geq 2$ be positive integers such that $(n_1, n_2)=1$. Then $x\not\in \langle n_1, n_2 \rangle$ if and only if $x=n_1n_2-an_1-bn_2$ for some $a, b\in \N$.
\end{proposition}

\subsection{Function Fields} Let $\cX$ be a nonsingular, projective, absolutely irreducible algebraic curve over $K$ with genus $g(\cX)$ and $K(\cX)$ be the function field of $\cX$. For each place $Q\in \cP_{K(\cX)}$, the Weierstrass semigroup $H(Q)$ has the structure of a numerical semigroup. Moreover, it is a well-known fact that for all but finitely many places $Q\in\cP_{K(\cX)}$, the gap set is always the same. This set is called the gap sequence of $\cX$. The places for which the gap set is not equal to the gap sequence of $\cX$ are called Weierstrass places.
 
Several upper bounds for the number of rational places of algebraic curves are available in the literature. The Hasse-Weil bound states that for a curve $\cX$ defined over $\fq$, 
$$
\# \cX(\fq)\leq q+1+2g(\cX)\sqrt{q}.
$$
The curve $\cX$ is called $\fq$-maximal if equality holds in the Hasse-Weil bound. Among other upper bounds for the number of rational places, we have the Lewittes bound \cite{L1990}.
\begin{theorem}[Lewittes bound]\label{Lbound} Let $\cX$ be a curve over $\fq$ and let $Q$ be
a rational place of $\cX$. Then
$$
\# \cX (\fq)\leq qm_{H(Q)}+1.
$$
\end{theorem}
For more on numerical semigroups and function fields, we refer to the books \cite{RG2009} and \cite{S2009} respectively.

\section{The semigroup $H(Q_{\infty})$}
Consider the algebraic curve
\begin{equation*}
\cX: \quad y^{m}=\prod_{i=1}^{r} (x-\a_i)^{\lambda_i}, \quad \la_i\in \N, \quad \text{and} \quad 1\leq \la_i < m,
\end{equation*}
where $m\geq 2$ and $r\geq 2$ are positive integers such that $p\nmid m$, $\al_1, \dots, \al_r\in K$ are pairwise distinct elements, $\la_0:=\sum_{i=1}^{r}\la_i$, and $(m, \la_0)=1$. By \cite[Proposition 3.7.3]{S2009}, this curve has genus 
\begin{equation}\label{genusX}
g(\cX)=\frac{(m-1)(r-1)+r-\sum_{i=1}^{r}(m, \la_i)}{2}.
\end{equation}

In this section, as one of our main results, we provide an explicit description of the Weierstrass semigroup $H(Q_{\infty})$ at the only place at infinity $Q_{\infty}$ of $\cX$. We start by recalling the property described in \cite[p.~94]{GKP1994}, which states that, for $m$ and $\la $ positive integers,
\begin{equation}\label{properties_floor_ceil}
\sum_{i=1}^{\la-1}\floor*{\frac{im}{\la}}=\frac{(m-1)(\la-1)+(m, \la)-1}{2}.
\end{equation}

To prove the main result of this section, we need the following technical lemma.
\begin{lemma}\label{lemma1}
Let $r, m, \la_0, \la_1, \la_2, \dots, \la_r$ be positive integers such that $\la_0=\sum_{i=1}^{r}\la_i$ and $r<\la_0$. For $k\in \{r, \dots, \la_0-1\}$, we define 
$$
\eta_k:=\max\left\{\rho_{s_1, \dots, s_r} \, : \, \sum_{i=1}^{r}s_i=k,\, 1\leq s_i\leq \la_i\right\}, \text{ where } \rho_{s_1, \dots, s_r} :=  \min_{1\leq i \leq r} \floor*{\frac{s_im}{\la_i}}.
$$
Then the sequence $\eta_r\leq \eta_{r+1}\leq  \dots\leq  \eta_{\la_0-1}$ is characterized by the following equality of multisets
\begin{equation}\label{multiseteq}
\Bigg\{\!\!\!\Bigg\{\eta_k: r\leq k\leq \la_0-1 \Bigg\}\!\!\!\Bigg\}=\Bigg\{\!\!\!\Bigg\{\floor*{\frac{s_im}{\la_i}}: 1\leq s_i<\la_i, \, 1\leq i\leq r\Bigg\}\!\!\!\Bigg\}.
\end{equation}
In particular, we have
\begin{equation*}
\sum_{k=r}^{\la_0-1}\eta_{k}= \frac{(m-1)(\la_0-r)-r+\sum_{i=1}^{r}(m, \la_i)}{2}.
\end{equation*}
\end{lemma}
\begin{proof}
First of all, note that, from the definition of $\eta_k$, we have that $\eta_k<m$ for each $k$. Furthermore, if $\eta_k=\rho_{u_1, \dots, u_r}=\Big\lfloor\frac{u_jm}{\la_j}\Big\rfloor$ for some $j$, where $\sum_{i=1}^{r}u_i=k$ and $r\leq k\leq \la_0-2$, then $u_j< \la_j$ and
\begin{equation*}
\eta_k=\rho_{u_1, \dots, u_r}\leq \rho_{u_1, \dots, u_j+1, \dots, u_r}\leq \eta_{k+1}.
\end{equation*}
This proves that $\eta_r\leq \eta_{r+1}\leq \dots\leq \eta_{\la_0-1}<m$ is a non-decreasing sequence. Let $S_1:=\{\!\!\{\eta_k: r\leq k\leq \la_0-1\}\!\!\}$ and $S_2:=\{\!\!\{\floor*{s_im/\la_i}: 1\leq s_i<\la_i, \, 1\leq i\leq r\}\!\!\}$. Now we are going to prove that $S_1=S_2$. From the definition of $\eta_k$, we have that $S_1^* \subseteq S_2^*$.
Furthermore, since the multisets $S_1$ and $S_2$ have the same cardinality, to prove that $S_1=S_2$ it is sufficient to show that $m_{S_1}(\eta_k)\leq m_{S_2}(\eta_k)$ for each $k$, that is, if $m_{S_1}(\eta_k)=n\geq 1$ then there exist distinct elements $j_1, j_2, \dots, j_n\in \{1, \dots, r\}$  and elements $s_{j_1}, s_{j_2}, \dots, s_{j_n}$ with $1\leq s_{j_i}\leq \la_{j_i}-1$ such that  
$$
\eta_k=\floor*{\frac{s_{j_1}m}{\la_{j_1}}}=\cdots =\floor*{\frac{s_{j_n}m}{\la_{j_n}}}.
$$
If $n=1$, there is nothing to prove, so we can assume that $n>1$. Without loss of generality, suppose that
\begin{equation}\label{aux_equa}
\eta_{k-1}<\eta_k=\eta_{k+1}=\dots=\eta_{k+n-1},
\end{equation}
where $\eta_{k-1}:=0$ if $k=r$. From the inclusion $S_1^*\subseteq S_2^*$, there exist $j_1\in \{1, \dots, r\}$ and $s_{j_1}\in \{1, \dots, \la_{j_1}-1\}$ such that $\eta_k=\Big\lfloor\frac{s_{j_1}m}{\la_{j_1}}\Big\rfloor$. Now, for each $i\in \{1, \dots, r\}$ we define the set 
$$
\Gamma_i:=\left\{s\in \N: \eta_k \leq \floor*{\frac{sm}{\la_i}}\text{ and } 1\leq s\leq\la_i\right\}.
$$ 
Next, we prove that $\Gamma_i\neq \emptyset$ for each $i$. Since $s_{j_1}< \la_{j_1}$, for $i\neq j_1$ we have that
$$
\floor*{\frac{s_{j_1}\la_i}{\la_{j_1}}}+1\leq \la_i \quad\text{and}\quad
\eta_k=\floor*{\frac{s_{j_1}m}{\la_{j_1}}}=\floor*{\left(\frac{s_{j_1}\la_i}{\la_{j_1}}\right)\frac{m}{\la_i}}\leq \floor*{\left(\floor*{\frac{s_{j_1}\la_i}{\la_{j_1}}}+1\right)\frac{m}{\la_i}},
$$
which implies that $\Big\lfloor\frac{s_{j_1}\la_i}{\la_{j_1}}\Big\rfloor +1\in \Gamma_i$ for $i\neq j_1$ and $s_{j_1}\in \Gamma_{j_1}$. Let $t_i$ be the smallest element of $\Gamma_i$. From definition of the set $\Gamma_{j_1}$, we have that $t_{j_1}\leq s_{j_1}$. If $t_{j_1}<s_{j_1}$ then
$$
1<\frac{m}{\la_{j_1}}\leq \frac{m}{\la_{j_1}}+\floor*{\frac{t_{j_1}m}{\la_{j_1}}}-\eta_k \leq  \frac{m}{\la_{j_1}}+\floor*{\frac{(s_{j_1}-1)m}{\la_{j_1}}}-\floor*{\frac{s_{j_1}m}{\la_{j_1}}}\leq \frac{s_{j_1}m}{\la_{j_1}}-\floor*{\frac{s_{j_1}m}{\la_{j_1}}},
$$
a contradiction, therefore $t_{j_1}=s_{j_1}$. Also, from definition of the sets $\Gamma_i$, we have that
\begin{equation*}
\floor*{\frac{(t_i-1)m}{\la_i}}<\eta_k=\rho_{t_1, \dots, t_r}\text{ for }i=1, \dots, r.
\end{equation*}
Note that $k=\sum_{i=1}^{r}t_i$. In fact, let $k':=\sum_{i=1}^{r}t_i$. By definition of $\eta_{k'}$, we have that $\eta_k=\rho_{t_1, \dots, t_r}\leq \eta_{k'}$, and from (\ref{aux_equa}), we deduce that $k\leq k'$. On the other hand, suppose that $(u_1, \dots, u_r)$ is an $r$-tuple such that $\eta_{k}=\rho_{u_1, \dots, u_r}$, $\sum_{i=1}^{r}u_i=k$, and $1\leq u_i\leq \la_i$. If there exists $j\in \{1, \dots, r\}$ such that $u_j<t_j$, then
$$
\eta_{k}=\rho_{u_1, \dots, u_r}=\min_{1\leq i \leq r}\floor*{\frac{u_im}{\la_i}}\leq \floor*{\frac{u_jm}{\la_j}}\leq \floor*{\frac{(t_j-1)m}{\la_j}}<\eta_k,
$$
a contradiction. Therefore $t_i\leq u_i$ for each $i=1, \dots, r$, and this implies that $k'\leq k$. Thus, we conclude that $k=k'=\sum_{i=1}^{r}t_i$. 

Now, we show that there exist distinct elements $j_2, \dots, j_n\in \{1, \dots, r\}\setminus \{j_1\}$ such that  
$$
\eta_k=\floor*{\frac{t_{j_1}m}{\la_{j_1}}}=\cdots =\floor*{\frac{t_{j_n}m}{\la_{j_n}}}.
$$
Suppose that $\eta_k<\Big\lfloor\frac{t_jm}{\la_j}\Big\rfloor$ for each $j\in \{1, \dots, r\}\setminus \{j_1\}$, then 
$\eta_k<\rho_{t_1, \dots, t_{j_1}+1, \dots, t_r}\leq \eta_{k+1}$ since $\sum_{i=1}^{r}t_i=k$. This is a contradiction to (\ref{aux_equa}). Therefore there exists $j_2\in \{1, \dots, r\}\setminus \{j_1\}$ satisfying
$$
\eta_k=\floor*{\frac{t_{j_1}m}{\la_{j_1}}}=\floor*{\frac{t_{j_2}m}{\la_{j_2}}}\quad \text{and} \quad t_{j_2}<\la_{j_2},
$$
where the strict inequality $t_{j_2}<\la_{j_2}$ follows from the fact that $\eta_k<m$. If $\eta_k<\Big\lfloor\frac{t_jm}{\la_j}\Big\rfloor$ for each $j\in \{1, \dots, r\}\setminus \{j_1, j_2\}$, then $
\eta_k<\rho_{t_1, \dots, t_{j_1}+1, \dots, t_{j_2}+1, \dots, t_r}\leq \eta_{k+2}
$, again a contradiction to (\ref{aux_equa}). Therefore there exists $j_3\in \{1, \dots, r\}\setminus \{j_1, j_2\}$ such that
$$
\eta_k=\floor*{\frac{t_{j_1}m}{\la_{j_1}}}=\floor*{\frac{t_{j_2}m}{\la_{j_2}}}=\floor*{\frac{t_{j_3}m}{\la_{j_3}}} \quad \text{and} \quad t_{j_3}<\la_{j_3}.
$$
By continuing this process, we obtain distinct elements $j_1, j_2, \dots, j_n$ such that
$$
\eta_k=\floor*{\frac{t_{j_1}m}{\la_{j_1}}}=\cdots =\floor*{\frac{t_{j_n}m}{\la_{j_n}}} \text{ and } t_{j_i}<\la_{j_i} \text{ for each }i=1, \dots, n.
$$
Finally, from (\ref{properties_floor_ceil}), we conclude that
\begin{align*}
\sum_{k=r}^{\la_0-1}\eta_k&=\sum_{i=1}^{r}\sum_{s=1}^{\la_i-1}\floor*{\frac{sm}{\la_i}}=\sum_{i=1}^{r}\frac{(m-1)(\la_i-1)-1+(m, \la_i)}{2}\\
&=\frac{(m-1)(\la_0-r)-r+\sum_{i=1}^{r}(m, \la_i)}{2}.
\end{align*}
\end{proof}
\begin{theorem}\label{theorem1}
Let $m \geq 2$ and $r\geq 2$ be  integers such that $p\nmid  m$. Let $\cX$ be the algebraic curve defined by the affine equation
\begin{equation}\label{curveX}
\cX: \quad y^{m}=\prod_{i=1}^{r} (x-\a_i)^{\lambda_i},\quad \la_i\in\N,\quad  \text{and} \quad 1\leq \la_i< m,
\end{equation}
where $\a_1, \dots, \a_r$ are pairwise distinct elements of $K$. Define $\la_0:=\sum_{i=1}^{r}\la_i$ and suppose that $(m, \la_0)=1$. Then the Weierstrass semigroup at the only place at infinity $Q_\infty\in \cP_{K(\cX)}$ is given by the disjoint union 
$$
H(Q_{\infty})=\langle m, \la_0\rangle \cupdot \bigcupdot_{k=r}^{\la_0-1}B_{k},
$$
where $B_{k}=\left\{mk-k'\la_0: k'=1, \dots, \eta_{k}\right\}$ and $\eta_k$ are defined as in Lemma \ref{lemma1}. In particular, 
\begin{equation}\label{generatorset}
H(Q_{\infty})=\langle m, \la_0,  mk-\la_0\eta_k: k=r, \dots, \la_0-1\rangle.
\end{equation}
\end{theorem}
\begin{proof}
Clearly the result holds if $r=\la_0$, therefore we can assume that $r<\la_0$. 
We start by computing some principal divisors in $K(\cX)$. Let $P_{\al_i}\in \cP_{K(x)}$ be the place corresponding to $\al_i\in K$. For $k\in \{r, \dots, \la_0-1\}$, let $s_1, \dots, s_r$ be positive integers such that $1\leq s_i\leq \la_i$ and $\sum_{i=1}^{r}s_i=k$. Then
$$
(x-\a_i)_{K(\cX)}=\frac{m}{(m, \la_i)}\sum_{\substack{Q|P_{\al_i}\\Q\in \cP_{K(\cX)}}}Q-mQ_\infty, \quad (y)_{K(\cX)}=\sum_{i=1}^{r}\frac{\la_i}{(m, \la_i)}\sum_{\substack{Q|P_{\al_i}\\Q\in \cP_{K(\cX)}}}Q-\la_0Q_\infty,
$$
and
$$
\left(\frac{\prod_{i=1}^{r}(x-\al_i)^{s_i}}{y^{\rho_{s_1, \dots, s_r}}}\right)_{K(\cX)}=\sum_{i=1}^{r}\frac{s_im-\la_i\rho_{s_1, \dots, s_r}}{(m, \la_i)}\sum_{\substack{Q|P_{\al_i}\\Q\in \cP_{K(\cX)}}}Q-\left(mk-\la_0\rho_{s_1, \dots, s_r}\right) Q_\infty.
$$
By the definition of $\eta_k$, we have that $0< mk-\la_0\eta_k\in H(Q_{\infty})$ for $r\leq k< \la_0$ and therefore
\begin{equation}\label{union}
\langle m, \la_0\rangle \cup \bigcup_{k=r}^{\la_0-1}B_{k}  \subseteq H(Q_{\infty}).
\end{equation}
Now, we prove that the union given in (\ref{union}) is disjoint. For $k\in \{r, \dots, \la_0-1\}$ and $k'\in \{1, \dots, \eta_k\} $, an element of $B_k$ can be written as
$$mk-k'\la_0=m\la_0-(\la_0-k)m-k'\la_0.
$$
Therefore, from Proposition \ref{prop_pre}, $B_k\cap \langle m, \la_0 \rangle = \emptyset$. On the other hand, we have that $B_{k_1}\cap B_{k_2}=\emptyset$ for $k_1\neq k_2$. In fact, if $mk_1-\la_0k_1'=mk_2-\la_0k_2'$ for $r\leq k_1, k_2<\la_0$, $1\leq k_1'\leq \eta_{k_1}$, and $1\leq k_2'\leq \eta_{k_2}$, then $m(k_1-k_2)=\la_0(k_1'-k_2')$. Since $(m, \la_0)=1$ and $2-\la_0\leq k_1-k_2\leq \la_0-2$, we conclude that $k_1=k_2$. 

Finally, we prove that equality holds in (\ref{union}). Since 
$$
g(\cX)=\frac{(m-1)(r-1)+r-\sum_{i=1}^{r}(m, \la_i)}{2}\quad \text{and}\quad g_{\langle m , \la_0 \rangle}=\frac{(m-1)(\la_0-1)}{2},
$$ 
from Lemma \ref{lemma1} we obtain that
$$
\# \left(\bigcupdot_{k=r}^{\la_0-1}B_{k}\right)=\sum_{k=r}^{\la_0-1}\eta_{k}=\frac{(m-1)(\la_0-r)-r+\sum_{i=1}^{r}(m, \la_i)}{2}=\# \left(H(Q_{\infty})\setminus \langle m, \la_0 \rangle\right) 
$$
and the result follows.
\end{proof}
In general, we have that a minimal system of generators of a numerical semigroup $H$ has cardinality at most the multiplicity of the semigroup, that is, $e_{H}\leq m_{H}$, see \cite[Proposition 2.10]{RG2009}. Since $m\in H(Q_{\infty})$, $e_{H(Q_{\infty})}\leq m_{H(Q_{\infty})}\leq m$. However, in general, it is difficult to obtain a minimal system of generators to $H(Q_{\infty})$ from the system of generators given in (\ref{generatorset}).

For example, for the curve $y^5=x(x-1)^2$
defined over $\fq$ with $5\nmid q$, the system of generators for the semigroup $H(Q_{\infty})$ provided by Theorem \ref{theorem1} is given  by $H(Q_{\infty})=\left\langle 3, 4, 5\right\rangle$ and therefore is a minimal system of generators. However, this does not happen in general. In fact, if $\eta_{k}=\eta_{k+1}$ for some $k$, then we can remove the element $m(k+1)-\la_0\eta_{k+1}$ of the system of generators given in (\ref{generatorset}) since $m(k+1)-\la_0\eta_{k+1}=mk-\la_0\eta_k+m$. More generally, define $\la:=\max_{1\leq i \leq r}\la_i$. If $\la=1$ then $H(Q_{\infty})=\langle m, \la_0\rangle$ and $e_{H(Q_{\infty})}=2$. If $\la>1$, then for $i\in \{\floor*{m/\la}, \dots, m-\ceil*{m/\la}\}$ define $k_i:=0$ if there is no $k\in \{r, \dots, \la_0-1\}$ such that $\eta_k=i$, and $k_i:=\min\{k: r\leq k< \la_0,\, \eta_k=i \}$ otherwise. Thus, for each $i$ such that $k_i\neq 0$ and $k$ such that $\eta_{k}=i$, we can write $mk-\la_0\eta_k=mk_i-\la_0\eta_{k_i}+m(k-k_i)$. Therefore, by removing the element $mk-\la_0\eta_k$ from the system of generators given in (\ref{generatorset}) we obtain that 
$$
H(Q_{\infty})=\left\langle m, \la_0, mk_i-\la_0\eta_{k_i}: i=\floor*{\frac{m}{\la}}, \dots, m-\ceil*{\frac{m}{\la}} \text{ and } k_i\neq 0 \right\rangle
$$
and $e_{H(Q_{\infty})}\leq m-\Big\lceil\frac{m}{\la}\Big\rceil-\Big\lfloor\frac{m}{\la}\Big\rfloor+3\leq m$.
\begin{example}[Plane model of the $GGS$ curve]\label{exam_GSS}
The $GGS$ curve is the first generalization of the $GK$ curve, which is the first example of a maximal curve not covered by the Hermitian curve, see \cite{GGS2010}. The $GGS$ curve is an $\mathbb{F}_{q^{2n}}$-maximal curve for $n\geq 3$ an odd integer, and it is described by the following  plane model: 
$$
y^{q^n+1}=(x^q+x)h(x)^{q+1},\text{ where } h(x)=\sum_{i=0}^{q}(-1)^{i+1}x^{i(q-1)}.
$$
This curve only has one place at infinity $Q_{\infty}$. In order to calculate the Weierstrass semigroup $H(Q_{\infty})$, note that $h(x)$ is a separable polynomial of degree $q(q-1)$. Using our standard notation as in Theorem \ref{theorem1}, we have that $m=q^n+1$, $r=q^2$, $\la_0=q^3$,
$\la_1=\dots =\la_{q}=1$, and $\la_{q+1}=\dots =\la_{q^2}=q+1$. From the characterization of the multiset $S=\{\!\!\{\eta_k: r\leq k\leq \la_0-1 \}\!\!\}$ given in Lemma \ref{lemma1}, we have that 
$$
S^*=\left\{\frac{(\be+1)(q^n+1)}{q+1}:0\leq \be\leq q-1\right\}.
$$
Furthermore, since $\la_1=\dots =\la_{q}=1$ and $\la_{q+1}=\dots =\la_{q^2}=q+1$, we have $m_S(a)=q^2-q$ for each $a\in S^*$. Thus, since $\eta_r\leq \eta_{r+1}\leq \dots\leq \eta_{\la_0-1}$ is a non-decreasing sequence, we obtain that
$$ 
\begin{matrix}
\eta_{r}&=&\eta_{r+1}&=&\dots & =&\eta_{r+q^2-q-1}&=&\frac{q^n+1}{q+1}\\
\eta_{r+q^2-q}&=&\eta_{r+q^2-q+1}&=&\dots  &=&\eta_{r+2(q^2-q)-1}&=&\frac{2(q^n+1)}{q+1}\\
&&&&\vdots\\
\eta_{r+\beta (q^2-q)}&=&\eta_{r+\beta (q^2-q)+1}&=&\dots  &=&\eta_{r+(\beta +1)(q^2-q)-1}&=&\frac{(\beta +1)(q^n+1)}{q+1}\\
&&&&\vdots\\
\eta_{r+(q-1) (q^2-q)}&=&\eta_{r+(q-1) (q^2-q)+1}&=&\dots  &=&\eta_{r+q(q^2-q)-1}&=&\frac{q(q^n+1)}{q+1}.
\end{matrix}
$$
Therefore, 
$$
\eta_{r+\be(q^2-q)+i}=\frac{(\be+1)(q^n+1)}{q+1}\,\text{ for }0\leq \be \leq q-1\text{ and }0\leq i\leq q^2-q-1.
$$
Moreover, since
$$
m(r+\be (q^2-q))-\la_0\eta_{r+\be (q^2-q)}=(q-\be)\frac{q(q^n+1)}{q+1}\, \text{ for }0\leq \be \leq q-1,
$$
it follows from Theorem \ref{theorem1} that
$$
H(Q_{\infty})=\left\langle q^n+1, q^3, \frac{q(q^n+1)}{q+1} \right\rangle.
$$
As expected, this description of $H(Q_{\infty})$ matches the result given in \cite[Corollary 3.5]{GCS2013}.
\end{example}
Let $n\geq 3$ be an odd integer, $m$ be a divisor of $q^n+1$, and $d$ be a divisor of $q+1$ such that $(m, d(q-1))=1$. In \cite[Theorem 3.1]{MQ2022}, the authors study the $\mathbb{F}_{q^{2n}}$-maximal curve defined by the affine equation
$$
\cY_{d, m}: \quad y^{m}=x^d(x^d-1)\left(\frac{x^{d(q-1)}-1}{x^d-1}\right)^{q+1}.
$$
This curve is a subcover of the second generalization of the $GK$ curve given by Beelen and Montanucci \cite{BM2018} and has only one place at infinity $Q_{\infty}$. In the following result, using Theorem \ref{theorem1}, we compute the Weierstrass semigroup $H(Q_{\infty})$.
\begin{proposition}\label{BMsubcover}
Let $n\geq 3$ be an odd integer, $m$ be a divisor of $q^n+1$, and $d$ be a divisor of $q+1$ such that $(m, d(q-1))=1$. Consider the curve 
$$
\cY_{d, m}: \quad y^{m}=x^d(x^d-1)\left(\frac{x^{d(q-1)}-1}{x^d-1}\right)^{q+1}.
$$
Then the Weierstrass semigroup at the only place at infinity $Q_{\infty}$ is given by
$$
H(Q_{\infty})=\left\langle m, \la_0, mk_{\be}-\la_0\floor*{\frac{(\be +1)m}{q+1}}: \be=0, \dots, q-1\right\rangle, 
$$
where $\la_0=dq(q-1)$ and $k_{\be}=d(q-1)(\be+1)+1+\floor*{\frac{\be d}{q+1}}-\be d$.
\end{proposition}
\begin{proof}
Using our standard notation, we have that $r=d(q-1)+1$, $\la_0=dq(q-1)$, $\la_1=d$, $\la_2=\dots=\la_{d+1}=1$, and $\la_{d+2}=\dots=\la_{d(q-1)+1}=q+1$. From the characterization of $S=\{\!\!\{\eta_k: r\leq k\leq \la_0-1 \}\!\!\}$ given in Lemma \ref{lemma1}, we obtain that 
$$
S^*=\left\{\floor*{\frac{(\be+1)m}{q+1}}:0\leq \be\leq q-1\right\}.
$$
Now, define $\delta_{\be}:=\ceil*{\frac{(\be+1)d}{q+1}}-\floor*{\frac{(\be+1)d}{q+1}}$ for $1\leq \be\leq q-1$. Since $\la_1=d$, $\la_2=\dots=\la_{d+1}=1$, and $\la_{d+2}=\dots=\la_{d(q-1)+1}=q+1$, we have
$$
m_S\left(\floor*{\frac{(\be+1)m}{q+1}}\right)=\left\{\begin{array}{ll}
d(q-2), & \text{if } \delta_{\be}=1,\\
d(q-2)+1, & \text{if } \delta_{\be}=0,
\end{array}\right.
$$
or, equivalently, 
\begin{equation}\label{multiex}
m_S\left(\floor*{\frac{(\be+1)m}{q+1}}\right)=d(q-2)+1-\delta_{\be}.
\end{equation}
In order to calculate the semigroup $H(Q_{\infty})$, let $k_{\be, i}:=r+\be (d(q-2)+1)-\sum_{j=0}^{\be-1}\delta_j+i$
for $0\leq \be \leq q-1$ and $0\leq i \leq d(q-2)-\delta_{\be}$. From (\ref{multiex}) and since $\eta_r\leq \eta_{r-1}\leq \dots \leq \eta_{\la_0-1}$ is a non-decreasing sequence, we obtain that
$$ 
\begin{matrix}
\eta_{r}&=& \eta_{r+1}&=&\dots & =&\eta_{r+d(q-2)-\delta_0}&=&\floor*{\frac{m}{q+1}}\\
\eta_{r+d(q-2)+1-\delta_0}&= &\eta_{r+d(q-2)+2-\delta_0}&=&\dots  &=&\eta_{r+2(d(q-2)+1)-1-\delta_0-\delta_1}&=&\floor*{\frac{2m}{q+1}}\\
&&&&\vdots\\
\eta_{k_{\beta, 0}}&=&\eta_{k_{\beta, 1}}&=&\dots  &=&\eta_{k_{\beta, d(q-2)-\delta_{\be}}}&=&\floor*{\frac{(\beta +1)m}{q+1}}\\
&&&&\vdots\\
\eta_{k_{q-1, 0}}&=&\eta_{k_{q-1, 1}}&=&\dots  &=&\eta_{k_{q-1, d(q-2)-\delta_{q-1}}}&=&\floor*{\frac{qm}{q+1}}.
\end{matrix}
$$
Therefore $\eta_{k_{\be, i}}=\floor*{\frac{(\be +1)m}{q+1}}$ for $0\leq \beta \leq q-1$ and $0\leq i\leq d(q-2)-\delta_{\be}$. From Theorem \ref{theorem1}, we conclude that
$$
H(Q_{\infty})=\left\langle m, \la_0, mk_{\be, 0}-\la_0\floor*{\frac{(\be +1)m}{q+1}}: \be=0, \dots, q-1\right\rangle.
$$
Now the proposition follows from the fact that $\be-\sum_{j=0}^{\be-1}\delta_j=\floor*{\frac{\be d}{q+1}}$ for $0\leq \be \leq q-1$.
\end{proof}

\section{The Frobenius number $F_{H(Q_{\infty})}$ and the Multiplicity $m_{H(Q_{\infty})}$}

With the explicit description of the Weierstrass semigroup $H(Q_{\infty})$ given in Theorem \ref{theorem1}, in this section we study the Frobenius number $F_{H(Q_{\infty})}$, the multiplicity $m_{H(Q_{\infty})}$, and the relationship between them. 

Henceforth, to simplify the notation, we define 
\begin{equation}\label{eta_epsilon}
\eta_s:=\left\{\begin{array}{ll}
0,& \mbox{if }0\leq s<r,\\
m-1, & \mbox{if }\la_0\leq s,
\end{array} \right . \quad \text{and}\quad 
\epsilon_k:=mk-\la_0(\eta_k+1) \text{ for }k\in \N_0.
\end{equation}
Thus, from Theorem \ref{theorem1}, we obtain that
\begin{equation}\label{Hinf}
H(Q_{\infty})=\langle \epsilon_k+\la_0: k=1, r, \dots, \la_0\rangle.
\end{equation}

We start by noticing that not all the elements $\epsilon_{r-1}, \dots, \epsilon_{\la_0-1}$ defined in (\ref{eta_epsilon}) are necessarily positive, however the following result states that the largest of them is equal to the Frobenius number $F_{H(Q_{\infty})}$. Moreover, we explicitly describe the gap set $G(Q_{\infty})$.
\begin{proposition}\label{prop_1} 
Using the same notation as in Theorem \ref{theorem1}, we have that
$$
F_{H(Q_{\infty})}=\max\{\epsilon_{r-1}, \dots, \epsilon_{\la_0-1}\}
$$
and 
$$
G(Q_{\infty})=\left\{ma-b\la_0: 1\leq a \leq \la_0-1, \, \eta_{a}+1\leq b\leq \floor*{\frac{am}{\la_0}}\right\}.
$$
\end{proposition}
\begin{proof}
From Theorem \ref{theorem1}, we have that
\begin{equation*}\label{eq_prop}
G(Q_{\infty})=\N\setminus \left(\langle m, \la_0\rangle \cupdot \bigcupdot_{k=r}^{\la_0-1}B_{k}\right)=\left(\N\setminus \langle m, \la_0\rangle\right)\setminus \left(\bigcupdot_{k=r}^{\la_0-1}B_{k}\right),
\end{equation*}
where $B_{k}=\{m\la_0-(\la_0-k)m-k'\la_0: k'=1, \dots,\eta_{k}\}$. Moreover, from Proposition \ref{prop_pre}, we know that the elements of $\N\setminus \langle m, \la_0\rangle$ are of the form $m\la_0-am-b\la_0$, where $a$ and $b$ are positive integers. Therefore,
$$
G(Q_{\infty})=\left\{m\la_0-am-b\la_0:  (a, b) \in \Delta \right\} \cap \N,
$$
where $\Delta= \{(a, b)\in \N^2: \eta_{\la_0-a}+1 \leq b\}$, and
$$F_{H(Q_{\infty})}=\max_{(a, b)\in \Delta} \{m\la_0-am-b\la_0\}.
$$  
By the definition of the set $\Delta$, $\max_{(a, b)\in \Delta}\{m\la_0-am-b\la_0\}$ is attained at a point in $\Delta$ of the form $(k, \eta_{\la_0-k}+1)$ for some $k\in \{1, \dots, \la_0-r+1\}$, see Figure \ref{figure1}. Thus, $F_{H(Q_{\infty})}=\max\{\epsilon_{r-1}, \dots, \epsilon_{\la_0-1}\}$. Moreover,
\begin{align*}
G(Q_{\infty})&=\left\{m\la_0-am-b\la_0:  (a, b) \in \Delta \right\} \cap \N\\
&=\left\{m(\la_0-a)-b\la_0:  1\leq a \leq \la_0-1, \, \eta_{\la_0-a}+1\leq b\right\} \cap \N\\
&=\left\{ma-b\la_0:  1\leq a \leq \la_0-1, \, \eta_{a}+1\leq b\leq \floor*{\frac{am}{\la_0}}\right\}.
\end{align*}
\end{proof}
\begin{figure}[h!]
\begin{tikzpicture}[line cap=round,line join=round,>=triangle 45,x=0.38cm,y=0.3cm]
\clip(-6.,0.) rectangle (34.,28.);
\draw [->,line width=0.4pt] (4.,3.) -- (4.,26.);
\draw [->,line width=0.4pt] (4.,3.) -- (30.,3.);
\draw (5.6,2.07) node[anchor=north west] {${\scriptscriptstyle 1}$};
\draw (7.4,2.07) node[anchor=north west] {${\scriptscriptstyle 2}$};
\draw (9.4,2.07) node[anchor=north west] {${\scriptscriptstyle 3}$};
\draw (24.7,2.07) node[anchor=north west] {${\scriptscriptstyle \lambda_0-r}$};
\draw (27.3,2.07) node[anchor=north west] {${\scriptscriptstyle \lambda_0-r+1}$};
\draw (-0.1,22.73602307339717) node[anchor=north west] {${\scriptscriptstyle \eta_{\lambda_0-1}+1}$};
\draw (0.1,4.633932744380933) node[anchor=north west] {${\scriptscriptstyle \eta_{r-1}+1}$};
\draw (-0.1,18.92191163498439) node[anchor=north west] {${\scriptscriptstyle \eta_{\lambda_0-2}+1}$};
\draw (-0.1,17.92191163498439) node[anchor=north west] {${\scriptscriptstyle \eta_{\lambda_0-3}+1}$};
\draw (0.9,8.016646593742554) node[anchor=north west] {${\scriptscriptstyle \eta_r+1}$};
\draw (1.54006831497784739,14.26796178323543) node[anchor=north west] {$\vdots$};
\draw (17.322239972734423,2.07) node[anchor=north west] {$\cdots $};
\draw [line width=0.05pt,dash pattern=on 3pt off 3pt] (4.,22.)-- (6.,22.);
\draw [line width=0.05pt,dash pattern=on 3pt off 3pt] (6.,22.)-- (5.949936444368475,3.);
\draw [line width=0.05pt,dash pattern=on 3pt off 3pt] (4.,18.)-- (8.,18.);
\draw [line width=0.05pt,dash pattern=on 3pt off 3pt] (4.,12.)-- (18.,12.);
\draw [line width=0.05pt,dash pattern=on 3pt off 3pt] (4.,4.)-- (28.,4.);
\draw [line width=0.05pt,dash pattern=on 3pt off 3pt] (8.,18.)-- (7.962385183135641,3.);
\draw [line width=0.05pt,dash pattern=on 3pt off 3pt] (10.,17.)-- (10.044575026336888,3.);
\draw [line width=0.05pt,dash pattern=on 3pt off 3pt] (14.,15.)-- (14.048786263262366,3.);
\draw [line width=0.05pt,dash pattern=on 3pt off 3pt] (18.,12.)-- (17.999608017028837,3.);
\draw [line width=0.05pt,dash pattern=on 3pt off 3pt] (26.,7.)-- (25.98133574930029,3.);
\draw [line width=0.05pt,dash pattern=on 3pt off 3pt] (28.,4.)-- (28.010136109342533,3.);
\draw [line width=0.05pt,dash pattern=on 3pt off 3pt] (4.,16.987199668999594)-- (10.,17.);
\draw [line width=0.05pt,dash pattern=on 3pt off 3pt] (4.,15.065178275275345)-- (14.,15.);
\draw [line width=0.05pt,dash pattern=on 3pt off 3pt] (4.,7.056755801424279)-- (26.,7.);
\begin{scriptsize}
\draw [fill=black] (28.,11.) circle (0.5pt);
\draw [fill=black] (28.,10.) circle (0.5pt);
\draw [fill=black] (28.,18.) circle (0.5pt);
\draw [fill=black] (28.,22.) circle (0.5pt);
\draw [fill=black] (28.,19.) circle (0.5pt);
\draw [fill=black] (28.,17.) circle (0.5pt);
\draw [fill=black] (28.,4.) circle (0.5pt);
\draw [fill=black] (28.,23.) circle (0.5pt);
\draw [fill=black] (28.,9.) circle (0.5pt);
\draw [fill=black] (28.,21.) circle (0.5pt);
\draw [fill=black] (28.,7.) circle (0.5pt);
\draw [fill=black] (28.,6.) circle (0.5pt);
\draw [fill=black] (28.,12.) circle (0.5pt);
\draw [fill=black] (28.,25.) circle (0.5pt);
\draw [fill=black] (28.,16.) circle (0.5pt);
\draw [fill=black] (28.,15.) circle (0.5pt);
\draw [fill=black] (28.,14.) circle (0.5pt);
\draw [fill=black] (28.,5.) circle (0.5pt);
\draw [fill=black] (28.,20.) circle (0.5pt);
\draw [fill=black] (28.,8.) circle (0.5pt);
\draw [fill=black] (28.,13.) circle (0.5pt);
\draw [fill=black] (28.,24.) circle (0.5pt);
\draw [fill=black] (26.,11.) circle (0.5pt);
\draw [fill=black] (26.,10.) circle (0.5pt);
\draw [fill=black] (26.,13.) circle (0.5pt);
\draw [fill=black] (26.,16.) circle (0.5pt);
\draw [fill=black] (26.,8.) circle (0.5pt);
\draw [fill=black] (26.,21.) circle (0.5pt);
\draw [fill=black] (26.,24.) circle (0.5pt);
\draw [fill=black] (26.,15.) circle (0.5pt);
\draw [fill=black] (26.,22.) circle (0.5pt);
\draw [fill=black] (26.,14.) circle (0.5pt);
\draw [fill=black] (26.,23.) circle (0.5pt);
\draw [fill=black] (26.,17.) circle (0.5pt);
\draw [fill=black] (26.,12.) circle (0.5pt);
\draw [fill=black] (26.,20.) circle (0.5pt);
\draw [fill=black] (26.,9.) circle (0.5pt);
\draw [fill=black] (26.,7.) circle (0.5pt);
\draw [fill=black] (26.,18.) circle (0.5pt);
\draw [fill=black] (26.,19.) circle (0.5pt);
\draw [fill=black] (26.,25.) circle (0.5pt);
\draw [fill=black] (8.,23.) circle (0.5pt);
\draw [fill=black] (8.,24.) circle (0.5pt);
\draw [fill=black] (8.,19.) circle (0.5pt);
\draw [fill=black] (8.,20.) circle (0.5pt);
\draw [fill=black] (8.,18.) circle (0.5pt);
\draw [fill=black] (8.,21.) circle (0.5pt);
\draw [fill=black] (8.,22.) circle (0.5pt);
\draw [fill=black] (8.,25.) circle (0.5pt);
\draw [fill=black] (6.,23.) circle (0.5pt);
\draw [fill=black] (6.,22.) circle (0.5pt);
\draw [fill=black] (6.,24.) circle (0.5pt);
\draw [fill=black] (6.,25.) circle (0.5pt);
\draw [fill=black] (24.,17.) circle (0.5pt);
\draw [fill=black] (24.,12.) circle (0.5pt);
\draw [fill=black] (24.,20.) circle (0.5pt);
\draw [fill=black] (24.,21.) circle (0.5pt);
\draw [fill=black] (24.,25.) circle (0.5pt);
\draw [fill=black] (24.,19.) circle (0.5pt);
\draw [fill=black] (24.,22.) circle (0.5pt);
\draw [fill=black] (24.,24.) circle (0.5pt);
\draw [fill=black] (24.,15.) circle (0.5pt);
\draw [fill=black] (24.,14.) circle (0.5pt);
\draw [fill=black] (24.,18.) circle (0.5pt);
\draw [fill=black] (24.,16.) circle (0.5pt);
\draw [fill=black] (24.,13.) circle (0.5pt);
\draw [fill=black] (24.,23.) circle (0.5pt);
\draw [fill=black] (10.,21.) circle (0.5pt);
\draw [fill=black] (10.,23.) circle (0.5pt);
\draw [fill=black] (10.,25.) circle (0.5pt);
\draw [fill=black] (10.,24.) circle (0.5pt);
\draw [fill=black] (10.,18.) circle (0.5pt);
\draw [fill=black] (10.,22.) circle (0.5pt);
\draw [fill=black] (10.,20.) circle (0.5pt);
\draw [fill=black] (10.,17.) circle (0.5pt);
\draw [fill=black] (10.,19.) circle (0.5pt);
\draw [fill=black] (16.,25.) circle (0.5pt);
\draw [fill=black] (12.,20.) circle (0.5pt);
\draw [fill=black] (12.,19.) circle (0.5pt);
\draw [fill=black] (14.,24.) circle (0.5pt);
\draw [fill=black] (14.,21.) circle (0.5pt);
\draw [fill=black] (16.,24.) circle (0.5pt);
\draw [fill=black] (16.,23.) circle (0.5pt);
\draw [fill=black] (20.,22.) circle (0.5pt);
\draw [fill=black] (16.,17.) circle (0.5pt);
\draw [fill=black] (20.,24.) circle (0.5pt);
\draw [fill=black] (18.,20.) circle (0.5pt);
\draw [fill=black] (22.,25.) circle (0.5pt);
\draw [fill=black] (14.,22.) circle (0.5pt);
\draw [fill=black] (14.,23.) circle (0.5pt);
\draw [fill=black] (18.,21.) circle (0.5pt);
\draw [fill=black] (18.,18.) circle (0.5pt);
\draw [fill=black] (20.,21.) circle (0.5pt);
\draw [fill=black] (22.,23.) circle (0.5pt);
\draw [fill=black] (22.,19.) circle (0.5pt);
\draw [fill=black] (20.,18.) circle (0.5pt);
\draw [fill=black] (20.,23.) circle (0.5pt);
\draw [fill=black] (16.,18.) circle (0.5pt);
\draw [fill=black] (20.,17.) circle (0.5pt);
\draw [fill=black] (22.,20.) circle (0.5pt);
\draw [fill=black] (18.,23.) circle (0.5pt);
\draw [fill=black] (22.,18.) circle (0.5pt);
\draw [fill=black] (22.,22.) circle (0.5pt);
\draw [fill=black] (16.,21.) circle (0.5pt);
\draw [fill=black] (20.,19.) circle (0.5pt);
\draw [fill=black] (14.,25.) circle (0.5pt);
\draw [fill=black] (16.,20.) circle (0.5pt);
\draw [fill=black] (14.,20.) circle (0.5pt);
\draw [fill=black] (12.,25.) circle (0.5pt);
\draw [fill=black] (18.,25.) circle (0.5pt);
\draw [fill=black] (18.,24.) circle (0.5pt);
\draw [fill=black] (12.,22.) circle (0.5pt);
\draw [fill=black] (18.,17.) circle (0.5pt);
\draw [fill=black] (12.,21.) circle (0.5pt);
\draw [fill=black] (20.,20.) circle (0.5pt);
\draw [fill=black] (16.,22.) circle (0.5pt);
\draw [fill=black] (14.,17.) circle (0.5pt);
\draw [fill=black] (12.,23.) circle (0.5pt);
\draw [fill=black] (22.,21.) circle (0.5pt);
\draw [fill=black] (18.,19.) circle (0.5pt);
\draw [fill=black] (12.,18.) circle (0.5pt);
\draw [fill=black] (22.,24.) circle (0.5pt);
\draw [fill=black] (14.,18.) circle (0.5pt);
\draw [fill=black] (16.,19.) circle (0.5pt);
\draw [fill=black] (14.,19.) circle (0.5pt);
\draw [fill=black] (22.,17.) circle (0.5pt);
\draw [fill=black] (12.,17.) circle (0.5pt);
\draw [fill=black] (18.,22.) circle (0.5pt);
\draw [fill=black] (12.,24.) circle (0.5pt);
\draw [fill=black] (20.,25.) circle (0.5pt);
\draw [fill=black] (14.,15.) circle (0.5pt);
\draw [fill=black] (16.,15.) circle (0.5pt);
\draw [fill=black] (22.,16.) circle (0.5pt);
\draw [fill=black] (22.,15.) circle (0.5pt);
\draw [fill=black] (18.,15.) circle (0.5pt);
\draw [fill=black] (14.,17.) circle (0.5pt);
\draw [fill=black] (20.,16.) circle (0.5pt);
\draw [fill=black] (18.,16.) circle (0.5pt);
\draw [fill=black] (16.,16.) circle (0.5pt);
\draw [fill=black] (14.,16.) circle (0.5pt);
\draw [fill=black] (22.,17.) circle (0.5pt);
\draw [fill=black] (20.,15.) circle (0.5pt);
\draw [fill=black] (16.,17.) circle (0.5pt);
\draw [fill=black] (20.,17.) circle (0.5pt);
\draw [fill=black] (18.,17.) circle (0.5pt);
\draw [fill=black] (22.,13.) circle (0.5pt);
\draw [fill=black] (20.,13.) circle (0.5pt);
\draw [fill=black] (18.,14.) circle (0.5pt);
\draw [fill=black] (22.,14.) circle (0.5pt);
\draw [fill=black] (22.,12.) circle (0.5pt);
\draw [fill=black] (18.,13.) circle (0.5pt);
\draw [fill=black] (18.,12.) circle (0.5pt);
\draw [fill=black] (20.,12.) circle (0.5pt);
\draw [fill=black] (20.,14.) circle (0.5pt);
\draw [fill=black] (18.,12.) ++(-1.5pt,0 pt) -- ++(1.5pt,1.5pt)--++(1.5pt,-1.5pt)--++(-1.5pt,-1.5pt)--++(-1.5pt,1.5pt);
\draw [fill=black] (28.,4.) ++(-1.5pt,0 pt) -- ++(1.5pt,1.5pt)--++(1.5pt,-1.5pt)--++(-1.5pt,-1.5pt)--++(-1.5pt,1.5pt);
\draw [fill=black] (14.,15.) ++(-1.5pt,0 pt) -- ++(1.5pt,1.5pt)--++(1.5pt,-1.5pt)--++(-1.5pt,-1.5pt)--++(-1.5pt,1.5pt);
\draw [fill=black] (10.,17.) ++(-1.5pt,0 pt) -- ++(1.5pt,1.5pt)--++(1.5pt,-1.5pt)--++(-1.5pt,-1.5pt)--++(-1.5pt,1.5pt);
\draw [fill=black] (8.,18.) ++(-1.5pt,0 pt) -- ++(1.5pt,1.5pt)--++(1.5pt,-1.5pt)--++(-1.5pt,-1.5pt)--++(-1.5pt,1.5pt);
\draw [fill=black] (6.,22.) ++(-1.5pt,0 pt) -- ++(1.5pt,1.5pt)--++(1.5pt,-1.5pt)--++(-1.5pt,-1.5pt)--++(-1.5pt,1.5pt);
\draw [fill=black] (26.,7.) ++(-1.5pt,0 pt) -- ++(1.5pt,1.5pt)--++(1.5pt,-1.5pt)--++(-1.5pt,-1.5pt)--++(-1.5pt,1.5pt);
\end{scriptsize}
\end{tikzpicture}
\caption{Description of the set $\Delta$}\label{figure1}
\end{figure}
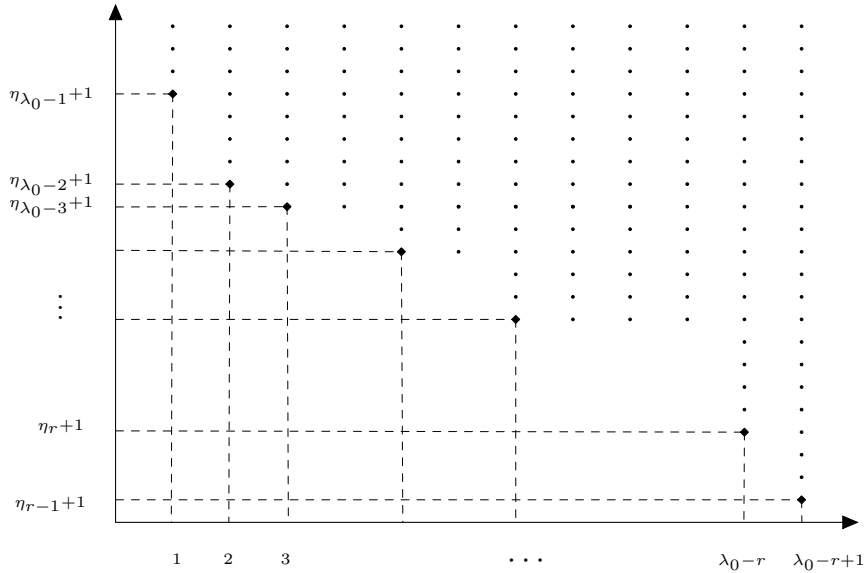
Now, we provide sufficient conditions to determine whether the semigroup $H(Q_{\infty})$ is symmetric. For this, we need a remark and a lemma.
\begin{remark}\label{remark_imp}
Due to the characterization of the sequence $\eta_{r}\leq \eta_{r+1}\leq \dots \leq \eta_{\la_0-1}$ given in Lemma \ref{lemma1}, we can see that, for $s\in \N_0$, $\eta_{s}+\eta_{r+\la_0-1-s}=m$ or $\eta_{s}+\eta_{r+\la_0-1-s}=m-1$. In fact, if $0\leq s\leq r-1$ or $\la_0\leq s$ the assertion is clear. Let $k\in \{r, \dots, \la_0-1\}$ and  $n\in \N$ be such that
$$
\eta_{k-1}<\eta_{k}=\eta_{k+1}=\dots =\eta_{k+n-1}< \eta_{k+n}.
$$
From Lemma \ref{lemma1}, there exist exactly $n$ distinct elements $j_1, \dots, j_n\in \{1, \dots, r\}$ and positive integers $s_{j_1}, \dots, s_{j_n}$ such that $1\leq s_{j_i}< \la_{j_i}$ and
$$
\eta_{k}=\floor*{\frac{s_{j_1}m}{\la_{j_1}}}=\floor*{\frac{s_{j_2}m}{\la_{j_2}}}=\dots =\floor*{\frac{s_{j_n}m}{\la_{j_n}}}.
$$
Without loss of generality, we can assume that 
$$
\ceil*{\frac{s_{j_{1}}m}{\la_{j_{ 1}}}}\leq \ceil*{\frac{s_{j_{2}}m}{\la_{j_{2}}}} \leq \dots \leq \ceil*{\frac{s_{j_{n}}m}{\la_{j_{ n}}}}
$$
and therefore
$$
\floor*{\frac{(\la_{j_{n}}-s_{j_{n}})m}{\la_{j_{n}}}}\leq \floor*{\frac{(\la_{j_{n-1}}-s_{j_{n-1}})m}{\la_{j_{n-1}}}}\leq \dots \leq \floor*{\frac{(\la_{j_{1}}-s_{j_{1}})m}{\la_{j_{1}}}}.
$$
This leads to
$$
\eta_{r+\la_0-1-(k+i)}=\floor*{\frac{(\la_{j_{i+1}}-s_{j_{i+1}})m}{\la_{j_{i+1}}}} \text{ for }i=0, \dots, n-1
$$
and, consequently,
$$
\eta_{k+i}+\eta_{r+\la_0-1-(k+i)}=\floor*{\frac{s_{j_{i+1}}m}{\la_{j_{i+1}}}}+\floor*{\frac{(\la_{j_{i+1}}-s_{j_{i+1}})m}{\la_{j_{i+1}}}}=m-\left(\ceil*{\frac{s_{j_{i+1}}m}{\la_{j_{i+1}}}}-\floor*{\frac{s_{j_{i+1}}m}{\la_{j_{i+1}}}}\right)
$$
for $i=0, \dots, n-1$. In particular, if $(m, \la_j)=1$ for each $j$, we obtain that $\eta_{s}+\eta_{r+\la_0-1-s}=m-1$ for $s\in \N_0$, and if $\la_j$ divides $m$ for each $j$, we obtain that $\eta_{s}+\eta_{r+\la_0-1-s}=m$ for $s=r, \dots, \la_0-1$.
\end{remark}
\begin{lemma}\label{lemma_2}
For $k\in \N_0$, the following statements hold:
\begin{enumerate}[i)] 
\item If $\eta_k+\eta_{r+\la_0-1-k}=m$, then $\epsilon_k+\epsilon_{r+\la_0-1-k}=\epsilon_{r-1}-\la_0$ and $\epsilon_{r-1}> \epsilon_k$.
\item If $\eta_k+\eta_{r+\la_0-1-k}=m-1$, then $\epsilon_k+\epsilon_{r+\la_0-1-k}=\epsilon_{r-1}$, and $\epsilon_{r-1}> \epsilon_k$ if and only if $0 <\epsilon_{r+\la_0-1-k}$. 
\item $\epsilon_k < 0$ if and only if $\eta_{k}=\floor*{\frac{km}{\la_0}}$.
\end{enumerate}
\end{lemma}
\begin{proof}
$i)$ It is enough to note that
\begin{align*}
\epsilon_{r+\la_0-1-k}&=m(r+\la_0-1-k)-\la_0(\eta_{r+\la_0-1-k}+1) \\ 
&=m(r+\la_0-1-k)-\la_0\left(m-\eta_k+1\right) \\
&=m(r-1)-\la_0-mk+\la_0\eta_k  \\
&=\epsilon_{r-1}-\epsilon_k-\la_0. 
\end{align*}
Therefore, $\epsilon_{r-1}-\epsilon_k=\epsilon_{r+\la_0-1-k}+\la_0> 0$. 

$ii)$ Similar to item $i)$.

$iii)$ Since $mk=\la_0\eta_k + (mk-\la_0\eta_k)$ and $0\leq mk-\la_0\eta_k$, we conclude that $\eta_k=\floor*{km/\la_0}$ if and only if $mk-\la_0\eta_k<\la_0$.
\end{proof}
\begin{theorem}\label{teo_2}
With the same notation as in Theorem \ref{theorem1},
$$F_{H(Q_{\infty})}=\epsilon_{r-1}\text{ and }H(Q_{\infty}) \text{ is symmetric}\quad \Leftrightarrow \quad \la_j\mid m \text{ for each }j=1, \dots, r.
$$
\end{theorem}
\begin{proof}
Suppose that $H(Q_\infty)$ is symmetric and $F_{H(Q_\infty)}=\epsilon_{r-1}$. From (\ref{genusX}) we obtain that
$$
F_{H(Q_{\infty})}=m(r-1)-\la_0=m(r-1)-\sum_{j=1}^{r}(m, \la_j).
$$
This implies that $\la_j$ divides $m$ for each $j=1, \dots, r$. 

Conversely, assume that $\la_j$ divides $m$ for each $j=1, \dots, r$. From Remark \ref{remark_imp} we have that $\eta_k+\eta_{r+\la_0-1-k}=m$ for $k=r, \dots, \la_0-1$, and from item $i)$ of Lemma \ref{lemma_2}, $\epsilon_{r-1}>\epsilon_k$ for $k=r, \dots, \la_0-1$. Therefore, from Proposition \ref{prop_1}, $F_{H(Q_{\infty})}=\max\{\epsilon_{r-1}, \dots, \epsilon_{\la_0-1}\}=\epsilon_{r-1}$ and 
$$
2g(\cX)-1=m(r-1)-\sum_{i=j}^{r}(m, \la_j)=m(r-1)-\la_0=\epsilon_{r-1}=F_{H(Q_{\infty})}.
$$
\end{proof}
\begin{example} From Example \ref{exam_GSS}, we know that the Weierstrass semigroup at the only place at infinity of the $GGS$ curve is given by $H(Q_{\infty})=\langle q^n+1, q^3, q(q^n+1)/(q+1)\rangle$. Therefore, we can determine if $H(Q_{\infty})$ is symmetric and we can calculate the Frobenius number $F_{H(Q_{\infty})}$. However, due to Theorem \ref{teo_2}, it is possible to know this without computing the semigroup $H(Q_{\infty})$ explicitly. In fact, since $q+1$ divides $q^n+1$, $H(Q_{\infty})$ is symmetric and 
$$
F_{H(Q_{\infty})}=(q^n+1)(q^2-1)-q^3=q^{n+2}-q^n-q^3+q^2-1.
$$ 
\end{example}
Next, we improve Proposition \ref{prop_1} to compute the Frobenius number $F_{H(Q_{\infty})}$ and establish a relationship between $F_{H(Q_{\infty})}$ and the multiplicity $m_{H(Q_{\infty})}$.
\begin{proposition}\label{prop_3}
Using the same notation as in Theorem \ref{theorem1}, the following statements hold:
\begin{enumerate}[i)]
\item $F_{H(Q_{\infty})}=\epsilon_{r-1}$ if and only if $\eta_s<\floor*{sm/\la_0}$ for each $s\in \{r, \dots, \la_0-1\}$ such that $\eta_s+\eta_{r+\la_0-1-s}=m-1$.
\item $F_{H(Q_{\infty})}=\max_{r-1\leq k < \la_0}\left\{\epsilon_{k}: \eta_{k}=\floor*{\frac{(k+1-r)m}{\la_0}}\right\}$.
\item If $(m, \la_j)=1$ for each $j=1, \dots, r$, then $m_{H(Q_{\infty})}=\min\{m, m(r-1)-F_{H(Q_{\infty})}\}$. 
\item If $\la_j$ divides $m$ for each $j=1, \dots, r$, then $m_{H(Q_{\infty})}=\min\left\{m, \la_0, \epsilon_{r-1}-\max_{r\leq k <\la_0}\epsilon_k\right\}$. 
\end{enumerate}
\end{proposition}
\begin{proof}
$i)$ It follows from Lemma \ref{lemma_2} and the fact that $\eta_s\leq \floor*{sm/\la_0}$ for all $s\in \N_0$.

$ii)$ It is enough to note that, from Lemma \ref{lemma_2}, we can rewrite the Frobenius number $F_{H(Q_{\infty})}$ as
\begin{align*}\label{eq_Frob}
F_{H(Q_{\infty})}&=\max_{r\leq k < \la_0}\left\{\epsilon_{r-1}, \epsilon_{k}: \epsilon_{r+\la_0-1-k}<0, \, \eta_{k}+\eta_{r+\la_0-1-k}=m-1\right\}\\
&=\max_{r\leq k < \la_0}\left\{\epsilon_{r-1}, \epsilon_{k}: \eta_{r+\la_0-1-k}=\floor*{\frac{(r+\la_0-1-k)m}{\la_0}}, \, \eta_{k}+\eta_{r+\la_0-1-k}=m-1\right\}\\
&=\max_{r\leq k < \la_0}\left\{\epsilon_{r-1}, \epsilon_{k}: \eta_{k}=\floor*{\frac{(k+1-r)m}{\la_0}}\right\}\\
&=\max_{r-1\leq k < \la_0}\left\{\epsilon_{k}: \eta_{k}=\floor*{\frac{(k+1-r)m}{\la_0}}\right\}.
\end{align*}
$iii)$ From (\ref{Hinf}) and Lemma \ref{lemma_2}, we obtain that 
\begin{align*}
m_{H(Q_{\infty})}&=\min\left\{m, \la_0, \la_0+\min_{r\leq k <\la_0}\epsilon_k\right\}\\
&=\min\left\{m, \la_0, \la_0+\min_{r\leq k <\la_0}\{\epsilon_{r-1}-\epsilon_{r+\la_0-1-k}\}\right\}\\
&=\min\left\{m, \la_0, \la_0+\epsilon_{r-1}-\max_{r\leq k <\la_0}\epsilon_{r+\la_0-1-k}\right\}\\
&=\min\left\{m, \la_0, \la_0+\epsilon_{r-1}-\max_{r\leq k <\la_0}\epsilon_{k}\right\}\\
&=\min\left\{m, m(r-1)-F_{H(Q_{\infty})}\right\}.
\end{align*}
$iv)$ Similar to the proof of item $iii)$.
\end{proof}
Next, we observe that for the curve $\cX$ defined in (\ref{curveX}), the elements of the set $\{\epsilon_k+\la_0: k=0, \dots, \la_0-1\}\subseteq H(Q_{\infty})$ form a complete set of representatives for the congruence classes of $\Z$ modulo $\la_0$ and 
$$\sum_{k=0}^{\la_0-1}\floor*{\frac{\epsilon_k+\la_0}{\la_0}}=g(\cX).
$$
Therefore, from Proposition \ref{prop2_pre}, the Apéry set of  $\la_0$ in the Weierstrass semigroup $H(Q_{\infty})$ is given by 
$$\Ap(H(Q_{\infty}), \la_0)=\left\{\epsilon_k+\la_0: k=0, \dots, \la_0-1\right\}.$$
We use this description of the Apéry set $\Ap(H(Q_{\infty}), \la_0)$ to characterize the symmetric Weierstrass semigroups $H(Q_{\infty})$ when $(m, \la_j)=1$ for each $j=1, \dots, r$.
\begin{theorem}\label{teo_sym}
Suppose that $(m, \la_j)=1$ for $j=1, \dots, r$. Then the followings statements are equivalent:
\begin{enumerate}[i)]
\item $H(Q_{\infty})=\langle m, r \rangle$.
\item $\la_1=\la_2=\dots =\la_r$.
\end{enumerate}
If in addition $r<m$, then all these statements are equivalent to the following:
\begin{enumerate}[i)]
\item[iii)] $H(Q_{\infty})$ is symmetric. 
\end{enumerate}
\end{theorem}
\begin{proof}
Clearly the result holds if $r=\la_0$. Suppose that $r<\la_0$.

$i) \Rightarrow ii):$ We start by proving that $r$ divides $\la_0$. In fact, since $\la_0, mr-\la_0\in H(Q_{\infty})=\langle m, r\rangle$, there exist $\al, \al', \tau, \tau'\in \N_{0}$, where $\tau, \tau'\leq m-1$ and $\tau\neq 0$, such that $\la_0=\al m+\tau r$ and $mr-\la_0=\al'm+\tau' r$. Therefore $m(r-\al-\al')=r(\tau+\tau')$. Since $H(Q_{\infty})=\langle m, r \rangle$, $(m, r)=1$ and therefore $m$ divides $\tau+\tau'$, where $1\leq \tau+\tau'\leq 2m-2$. This implies that $\tau+\tau'=m$ and $\al=-\al'$. It follows that $\al=\al'=0$ and $\la_0=\tau r$.

Now, let $\la:=\max_{1\leq i \leq r}\la_i$ and note that $\tau r =\la_0=\sum_{i=1}^{r}\la_i \leq \la r$, therefore $\tau\leq \la$. In the following, we prove that $\tau=\la$, which implies that $\la_1=\la_2=\dots=\la_r$.

For $\be\in \{1, \dots, \tau-1\}$ and $i\in \{0, \dots, r-1\}$
we have that
$$
\epsilon_{\beta r+i}+\la_0=mr-(r-i)m-(\tau\eta_{r\beta+i}-m\beta)r \in H(Q_{\infty})=\langle m, r \rangle.
$$
Therefore, from Proposition \ref{prop_pre}, it follows that 
\begin{equation}\label{inequality1}
\eta_{r\beta +i}\leq \floor*{\frac{\beta m}{\tau}} \text{ for }1\leq \be\leq  \tau-1 \text{ and }0\leq i \leq r-1.
\end{equation}
For $\be=1$ in (\ref{inequality1}) we obtain that 
$$
\Bigg\lfloor\frac{m}{\la}\Bigg\rfloor=\eta_{r}\leq \eta_{r+i}\leq \Bigg\lfloor\frac{m}{\tau}\Bigg\rfloor \,\text{ for }\,0\leq i\leq r-1,
$$
and for $\be=\tau-1$ and $i=r-1$ in (\ref{inequality1}),
$$
\quad m-\Bigg\lceil\frac{m}{\la}\Bigg\rceil=\floor*{\frac{(\la-1)m}{\la}}=\eta_{\la_0-1}=\eta_{r(\tau-1)+r-1}\leq \floor*{\frac{(\tau-1)m}{\tau}}=m-\Bigg\lceil\frac{m}{\tau}\Bigg\rceil.
$$
Since $(m, \la)=(m, \tau)=1$, then $\floor*{\frac{m}{\la}}=\floor*{\frac{m}{\tau}}$ and therefore $\eta_{r+i}=\floor*{\frac{m}{\la}}$ for $0\leq i\leq r-1$. Thus, from the characterization of the sequence $\eta_{r}\leq \eta_{r+1}\leq \dots \leq \eta_{\la_0-1}$ given in (\ref{multiseteq}), we have that 
$$
\eta_r=\floor*{\frac{m}{\la_1}}=\floor*{\frac{m}{\la_2}}=\dots=\floor*{\frac{m}{\la_r}}=\eta_{2r-1}
$$
and therefore $\eta_{2r}=\floor*{\frac{2m}{\la}}$. Moreover, from Remark \ref{remark_imp}, $\eta_{\la_0-1-i}=m-1-\eta_{r+i}=\Big\lfloor\frac{(\la-1)m}{\la}\Big\rfloor$ for $0\leq i \leq r-1$ and hence $\eta_{\la_0-r-1}=\Big\lfloor\frac{(\la-2)m}{\la}\Big\rfloor$.  

For $\be=2$ in (\ref{inequality1}) we have that 
$$
\floor*{\frac{2m}{\la}}= \eta_{2r}\leq \eta_{2r+i}\leq \floor*{\frac{2m}{\tau}}\, \text{ for }\,0\leq i\leq r-1,
$$
and for $\be=\tau-2$ and $i=r-1$ in (\ref{inequality1}),
$$
m-\ceil*{\frac{2m}{\la}}=\floor*{\frac{(\la-2)m}{\la}}=\eta_{\la_0-r-1}=\eta_{r(\tau-2)+r-1}\leq \floor*{\frac{(\tau-2)m}{\tau}}= m-\ceil*{\frac{2m}{\tau}}. 
$$
Similarly to the previous case, we deduce that $\floor*{\frac{2m}{\la}}=\floor*{\frac{2m}{\tau}}$, $\eta_{2r+i}=\floor*{\frac{2m}{\la}}$ and $\eta_{\la_0-r-1-i}=\Big\lfloor\frac{(\la-2)m}{\la}\Big\rfloor$ for  $0\leq i \leq r-1$. This implies that $\eta_{3r}=\floor*{\frac{3m}{\la}}$ and $\eta_{\la_0-2r-1}=\Big\lfloor\frac{(\la-3)m}{\la}\Big\rfloor$. 

By continuing this process, we obtain that
$$
\eta_{r\be+i}=\floor*{\frac{\be m}{\la}}\text{ for } 1\leq \be\leq \tau-1 \text{ and }0\leq i\leq r-1.
$$
In particular, for $\be=\tau-1$ and $i=r-1$ we have that 
$$
\floor*{\frac{(\tau-1)m}{\la}}=\eta_{r(\tau-1)+r-1}=\eta_{r\tau -1}=\eta_{\la_0-1}=\floor*{\frac{(\la-1)m}{\la}}.
$$
This implies that $\tau=\la$. 

$ii) \Rightarrow i):$ Suppose that $\la_1=\la_2=\dots=\la_r$. Then $\la_0=r\la_r$ and $\eta_{\be r +i }=\Big\lfloor\frac{\be m}{\la_r}\Big\rfloor$ for $1\leq \be\leq \la_r-1$ and $0\leq i\leq r-1$. On the other hand, from Theorem \ref{theorem1},
\begin{align*}
H(Q_{\infty})&=\left\langle m, r\la_r, r\left( \be m -\la_r \floor*{\frac{\be m}{\la_r}}\right): \be =1, \dots, \la_r-1\right\rangle\\
&=\left\langle m, r\la_r, r\la_r\left\{\frac{\be m}{\la_r}\right\}: \be =1, \dots, \la_r-1\right\rangle.
\end{align*}
Since $(m, \la_r)=1$, there exists $\be'\in \{1, \dots, \la_r-1\}$ such that $\left\{\frac{\be' m}{\la_r}\right\}=\frac{1}{\la_r}$ and therefore $H(Q_{\infty})=\langle m, r \rangle$. 

Now, suppose that $r< m$.

$i) \Rightarrow iii):$ It is clear.

$iii) \Rightarrow i):$ We are going to prove that $(m, r)=1$. We start by noting two important facts. First, note that 
\begin{equation}\label{sym_eq0}
(\epsilon_k+\la_0)\equiv 0 \mod m\quad\text{if and only if}\quad 0\leq k \leq r-1.
\end{equation}
Second, since $r<m$ and $(m, \la_j)=1$ for each $j$, then $H(Q_{\infty})$ is symmetric if and only if $m_{H(Q_{\infty})}=r$. In fact, for this case we have that $g(\cX)=(m-1)(r-1)/2$. Furthermore, from item $iii)$ of Proposition \ref{prop_3}, $m_{H(Q_\infty)}=\min\{m, m(r-1)-F_{H(Q_\infty)}\}$. If $H(Q_\infty)$ is symmetric, then $F_{H(Q_\infty)}=2g(\mathcal{X})-1=m(r-1)-r$ and 
$$
m_{H(Q_\infty)}=\min\{m, m(r-1)-F_{H(Q_\infty)}\}=\min\{m, r\}=r.
$$
Conversely, if $m_{H(Q_\infty)}=r$ then $m(r-1)-F_{H(Q_\infty)}=r$ and therefore $F_{H(Q_\infty)}=2g(\mathcal{X})-1$. This implies that $H(Q_\infty)$ is symmetric.  

Let $\sigma$ be the permutation of the set $\{0, \dots, \la_0-1\}$ such that
$$
\Ap (H(Q_{\infty}), \la_0)=\{0=\epsilon_{\sigma(0)}+\la_0< \epsilon_{\sigma(1)}+\la_0<\dots <\epsilon_{\sigma(\la_0-1)}+\la_0  \}.
$$
Since  $(m, \la_j)=1$ for $j=1, \dots, r$ and $H(Q_{\infty})$ is symmetric, then $F_{H(Q_{\infty})}=\epsilon_{\sigma(\la_0-1)}=m(r-1)-r$. Thus, from Proposition \ref{prop3_pre}, we have that
\begin{equation}\label{sym_eq1}
\epsilon_{\sigma(i)}+\epsilon_{\sigma(\la_0-1-i)}=m(r-1)-\la_0-r \quad \text{for }i=0, \dots, \la_0-1.
\end{equation}
On the other hand, from Proposition \ref{lemma_2}, we know that
\begin{equation}\label{sym_eq2}
\epsilon_{\sigma(i)}+\epsilon_{r+\la_0-1-\sigma(i)}=m(r-1)-\la_0 \quad \text{for }i=0, \dots, \la_0-1.
\end{equation}
Let $\la>0$ and $0\leq r'<r$ be integers such that $\la_0=\la r +r'$, and $i_1\in \{0, \dots, \la_0-1\}$ be such that $\sigma(\la_0-1-i_1)=r-1$. Then, from (\ref{sym_eq1}), 
$$
\epsilon_{\sigma(i_1)}=m(r-1)-\la_0-r-\epsilon_{\sigma(\la_0-1-i_1)}=m(r-1)-\la_0-r-\epsilon_{r-1}=-r.
$$
If $(\epsilon_{\sigma(i_1)}+\la_0)\equiv 0 \mod{m}$, then $m$ divides $\la_0-r$ and therefore $\la_0=ms+r$ for some integer $s$. Since $(m, \la_0)=1$, we conclude that $1=(m, \la_0)=(m, ms+r)=(m, r)$. Otherwise, from (\ref{sym_eq0}), $\sigma(i_1)\geq r$ and therefore there exists $i_2\in \{0, \dots, \la_0-1\}$ such that $\sigma(\la_0-1-i_2)=r+\la_0-1-\sigma(i_1)$. From (\ref{sym_eq1}) and (\ref{sym_eq2}), we have that
$$
\epsilon_{\sigma(i_2)}=m(r-1)-\la_0-r-\epsilon_{\sigma(\la_0-1-i_2)}=m(r-1)-\la_0-r-\epsilon_{r+\la_0-1-\sigma(i_1)}=\epsilon_{\sigma(i_1)}-r=-2r.
$$
If $(\epsilon_{\sigma(i_2)}+\la_0)\equiv 0 \mod{m}$, then $m$ divides $\la_0-2r$ and therefore $(m, r)=1$. Otherwise, $\sigma(i_2)\geq r$ and therefore there exists $i_3\in \{0, \dots, \la_0-1\}$ such that $\sigma(\la_0-1-i_3)=r+\la_0-1-\sigma(i_2)$ and 
$$
\epsilon_{\sigma(i_3)}=m(r-1)-\la_0-r-\epsilon_{\sigma(\la_0-1-i_3)}=m(r-1)-\la_0-r-\epsilon_{r+\la_0-1-\sigma(i_2)}=\epsilon_{\sigma(i_2)}-r=-3r.
$$
By continuing this process, we have that $(m, r)=1$ or we obtain a sequence $i_1, \dots, i_{\la}$ such that 
$$
\sigma(i_j)\geq r \quad \text{and} \quad \epsilon_{\sigma(i_j)}=-jr \quad \text{for }1\leq j\leq  \la.
$$ 
If the latter happens, then $0<\epsilon_{\sigma(i_{\la})}+\la_0=\la_{0}-\la r=r'<r$, a contradiction because $m_{H(Q_{\infty})}=r$. Therefore, $(m, r)=1$.  Finally,  since $\langle m, r\rangle\subseteq H(Q_{\infty})$ and $g(\cX)=(m-1)(r-1)/2$, we conclude that $H(Q_{\infty})=\langle m, r\rangle$.
\end{proof}

\section{Maximal Castle curves}
In this section, as an application of the results obtained, we characterize certain classes of $\fqs$-maximal Castle curves of type $(\cX, Q_{\infty})$ (that is, $\fqs$-maximal curves $\cX$ such that $\#\cX(\fqs)=q^2m_{H(Q_{\infty})}+1$ and $H(Q_{\infty})$ is symmetric), where $\cX$ is the curve defined by the equation $y^m=f(x)$, $f(x)\in \fqs[x]$ and $(m, \deg f)=1$, and $Q_{\infty}$ is the only place at infinity of the curve $\cX$. Some examples of $\fqs$-maximal Castle curves of this type are presented below:
\begin{itemize}
\item The Hermitian curve 
$$ y^{q+1}=x^q+x.$$
\item The curve over $\fqs$ defined by the affine equation 
$$
y^{q+1}=a^{-1}(x^{q/p}+x^{q/p^2}+\cdots +x^p+x),$$
where $p=\char (\fq)$ and $a\in \fqs$ is such that $a^{q}+a=0$ and $a\neq 0$.
\end{itemize}
Note that, in all cases, the places corresponding to the roots of the polynomial $f(x)$ are totally ramified in the extension $\fqs(x, y)/\fqs(x)$, the multiplicities of the roots of $f(x)$ are equal and $m=q+1$. We will show that, under certain conditions, all $\fqs$-maximal Castle curves of type $(\cX, Q_{\infty})$ have these characteristics.
\begin{lemma}
Let $\cX$ be the algebraic curve given in Theorem \ref{theorem1} and let $Q_{\infty}$ be its only place at infinity. Suppose that $\cX$ is defined over $\fqs$, $(m, \la_i)=1$ for $i=1, \dots, r$, $(\cX, Q_{\infty})$ is a Castle curve, and $r<m$. Then
$$
\cX \text{ is }\fqs\text{-maximal if and only if }m=q+1.
$$
\end{lemma}
\begin{proof}
From the assumptions, we obtain that $g(\cX)=(m-1)(r-1)/2$. Since $(\cX, Q_{\infty})$ is a Castle curve, $H(Q_{\infty})$ is symmetric and therefore $F_{H(Q_{\infty})}=2g(\cX)-1=mr-m-r$. Moreover, from $iii)$ of Proposition \ref{prop_3}, $m_{H(Q_{\infty})}=\min\{m, r\}=r$. Therefore, $\cX$ is $\fqs$-maximal if and only if 
\begin{equation*}\label{eq_Castle}
\#\cX(\fqs)=q^2r+1=q^2+1+q(m-1)(r-1).
\end{equation*}
Thus, the result follows.
\end{proof}
\begin{lemma}
Let $\cX$ be the algebraic curve given in Theorem \ref{theorem1} and let $Q_{\infty}$ be its only place at infinity. Suppose that $\cX$ is defined over $\fqs$, $m=q+1$, $r<q+1$, $(q+1, \la_i)=1$ for $i=1, \dots, r$, and $\cX$ is $\fqs$-maximal. The following statements are equivalent:
\begin{enumerate}[i)]
\item $H(Q_{\infty})$ is symmetric.
\item $\# \cX(\fqs)=q^2m_{H(Q_{\infty})}+1$.
\item $\la_1=\dots= \la_r$. 
\end{enumerate}
\end{lemma}
\begin{proof}
Note that from the hypotheses we have that $g(\mathcal{X})=q(r-1)/2$ and therefore $\#\mathcal{X}(\mathbb{F}_{q^2})=q^2+1+2g(\mathcal{X})q=q^2r+1$.

$i)\Leftrightarrow ii):$ It is enough to note that
\begin{align*}
H(Q_{\infty}) \text{ is symmetric } &\Leftrightarrow F_{H(Q_{\infty})}=qr-q-1\\
&\Leftrightarrow m_{H(Q_{\infty})}=r & \text{(from Proposition \ref{prop_3})}\\
&\Leftrightarrow \#\cX(\fqs)=q^2m_{H(Q_{\infty})}+1.
\end{align*}

$i)\Leftrightarrow iii):$ This follows directly from Theorem \ref{teo_sym}.
\end{proof}
We summarize these results in the following theorem.
\begin{theorem}\label{theorema3}
Let $\cX$ be the algebraic curve defined in Theorem \ref{theorem1} and let $Q_{\infty}$ be its only place at infinity. Suppose that $\cX$ is defined over $\fqs$, $(m, \la_i)=1$ for $i=1, \dots, r$, and $r<m$. Then the following statements are equivalent:
\begin{enumerate}[i)]
\item $(\cX, Q_{\infty})$ is a $\fqs$-maximal Castle curve.
\item $(\cX, Q_{\infty})$ is a Castle curve and $m=q+1$.
\item $\cX$ is $\fqs$-maximal, $H(Q_{\infty})$ is symmetric, and $m=q+1$. 
\item $\cX$ is $\fqs$-maximal, $\#\cX(\fqs)=q^2m_{H(Q_{\infty})}+1$, and $m=q+1$. 
\item $\cX$ is $\fqs$-maximal, $\la_1=\dots=\la_r$, and $m=q+1$. 
\end{enumerate}
\end{theorem}
Finally, we note that for the case when $\la_i$ divides $m$ for each $i=1, \dots, r$, the Weierstrass semigroup $H(Q_{\infty})$ is symmetric, see Theorem \ref{teo_2}. Therefore, by assuming  that $\cX$ is $\fqs$-maximal, we conclude that 
$$
(\cX, Q_{\infty})\text{ is } \fqs\text{-maximal Castle curve if and only if }\#\cX(\fqs)=q^2m_{H(Q_{\infty})}+1.
$$

\section{Acknowledgment}
I would like to thank Professors Luciane Quoos and Rohit Gupta, as well as the anonymous referees for their valuable comments and suggestions that improved the presentation of this paper.

\bibliographystyle{abbrv}
\bibliography{bibinfinityplace} 
\end{document}